\documentclass[11pt]{amsart}

\usepackage{amsmath, amsthm, amssymb, tikz, hyperref, enumerate,pgfkeys, pgfopts, xcolor,young,comment}
\usepackage{amsfonts}
\usepackage{ytableau}
\newcommand\observation{\begin{Observation}}
\newcommand\beq{\begin{equation}}
\newcommand\eeq{\end{equation}}
\newcommand\bce{\begin{center}}
\newcommand\ece{\end{center}}
\newcommand\bea{\begin{eqnarray}}
\newcommand\eea{\end{eqnarray}}
\newcommand\ba{\begin{array}}
\newcommand\ea{\end{array}}
\newcommand\ben{\begin{enumerate}}
\newcommand\een{\end{enumerate}}
\newcommand\bit{\begin{itemize}}
\newcommand\eit{\end{itemize}}
\newcommand\brr{\begin{array}}
\newcommand\err{\end{array}}
\newcommand\bt{\begin{tabular}}
\newcommand\et{\end{tabular}}

\renewcommand\S{{\mathcal S}}

\DeclareMathOperator\maj{maj}
\DeclareMathOperator\inv{inv}
\DeclareMathOperator\SYT{SYT}

\newcommand\x{{\mathbf x}}

\newcommand{\bbn}{\mathbb{N}}

\newcommand{\bbq}{\mathbb{Q}}

\newcommand{\F}{\mathcal{F}}

\DeclareMathOperator\Des{Des}

\DeclareMathOperator\ldes{ldes}
\DeclareMathOperator\stat{stat}

\DeclareMathOperator\ides{ides}
\DeclareMathOperator\imaj{imaj}

\DeclareMathOperator\bl{bl}

\DeclareMathOperator\Ltr{ltrMax}

\newcommand{\BBB}{{\mathcal{B}}}

\newcommand{\ch}{\operatorname{ch}}

\newcommand{\Q}{{\mathcal Q}}

\newtheorem{theorem}{Theorem}[section]

\newtheorem{fact}[theorem]{Fact}
\newtheorem{proposition}[theorem]{Proposition}
\newtheorem{lemma}[theorem]{Lemma}
\newtheorem{corollary}[theorem]{Corollary}
\newtheorem{definition}[theorem]{Definition}
\newtheorem{obs}[theorem]{Observation}
\newtheorem{exa}[theorem]{Example}
\newtheorem{example}[theorem]{Example}

\newtheorem{defn}[theorem]{Definition}

\newtheorem{remark}[theorem]{Remark}
\newtheorem{problem}[theorem]{Problem}

\numberwithin{figure}{section}

\title[Block decomposition and Schur-positivity]{Block decomposition of permutations\\ and Schur-positivity}

\author{Ron M.\ Adin}
\address{Department of Mathematics, Bar-Ilan University, Ramat-Gan 52900, Israel}
\email{radin@math.biu.ac.il}
\author{Eli Bagno}
\address{Department of Applied Mathematics, Jerusalem College of Technology, 21 Havaad Haleumi St., Jerusalem, Israel}
\email{bagnoe@g.jct.ac.il}
\author{Yuval Roichman}
\address{Department of Mathematics, Bar-Ilan University, Ramat-Gan 52900, Israel}
\email{yuvalr@math.biu.ac.il}


\date{submitted: February 15, 2017; revised: August 28, 2017}

\begin{document}

\begin{abstract}


The {\em block number} of a permutation is the maximal number of
components in its expression as a direct sum. 
We show that, for $321$-avoiding permutations,
the set of left-to-right maxima has the same distribution 
when the block number is assumed to be $k$
as when the last descent of the inverse is assumed to be at position $n - k$.
This result is analogous to the Foata-Sch\"{u}tzenberger
equi-distribution theorem, and implies that the quasi-symmetric
generating function of descent set over $321$-avoiding
permutations with a prescribed number of blocks is
Schur-positive.

\end{abstract}

\maketitle

\tableofcontents

\section{Introduction}

Given any subset $A$ of the symmetric group $\S_n$, define the
quasi-symmetric function
\[
\Q(A) := \sum\limits_{\pi\in A} \F_{n,\Des(\pi)},
\]
where $\Des(\pi):=\{i:\ \pi(i)>\pi(i+1)\}$ is the {\em descent set} of $\pi$
and $\F_{n,D}$ (for $D \subseteq [n-1]$) are Gessel's {\em fundamental quasi-symmetric functions;} see Subsection~\ref{sec:prel_quasi} for more details. 
The following long-standing problem was first posed
in~\cite{Gessel-Reutenauer}.

\begin{problem}\label{prob:symmetric}
For which subsets $A\subseteq \S_n$ is $\Q(A)$ symmetric?
\end{problem}

A symmetric function is called {\em Schur-positive} if all the
coefficients in its expansion in the basis of Schur functions are nonnegative.
Determining whether a given symmetric function is
Schur-positive is a major problem in contemporary algebraic
combinatorics~\cite{Stanley_problems}.

Call a subset $A \subseteq \S_n$ {\em Schur-positive} if $\Q(A)$ is symmetric and Schur-positive. 
Classical examples of Schur-positive sets of permutations include 
inverse descent classes~\cite{Gessel}, Knuth classes~\cite{Gessel},
conjugacy classes~\cite[Theorem 5.5]{Gessel-Reutenauer}, and
permutations with a fixed inversion number~\cite[Prop. 9.5]{Adin-R}.

\medskip

New constructions of Schur-positive sets of permutations 
were described in~\cite{ER} and~\cite{Sagan_talk}. 
Inspired by these examples, Sagan and Woo posed the 
problem of finding Schur-positive pattern-avoiding sets~\cite{Sagan_talk}.



\medskip

The goal of this paper is to present a new example of a Schur-positive set of permutations which involves pattern-avoidance: the set of $321$-avoiding permutations having a prescribed number of blocks. We shall state that more explicitly.

\medskip

A permutation $\pi \in \S_n$ is {\em $321$-avoiding} if the sequence $(\pi(1), \ldots, \pi(n))$ contains no decreasing subsequence of length $3$. Denote by $\S_n(321)$ the set of $321$-avoiding
permutations in $\S_n$.

For a permutation $\pi\in \S_n$ let
\[
\bl(\pi)=|\{ i \,:\, (\forall j \le i)\ \ \pi(j) \le i \}|
\]
be the {\em block number} of $\pi$. It is equal to the maximal number of summands in an expression of $\pi$ as a direct sum of permutations; see Subsection~\ref{sec:block} below. 
The block number was studied in~\cite{St05} (and references therein)
as the cardinality of the {\em connectivity set} of $\pi$.

Denote
\[
Bl_{n,k} := \{\pi\in \S_n(321) \,:\, \bl(\pi) = k\}.
\]
Recall the {\em Frobenius characteristic map} $\ch$, 
from class functions on $\S_n$ to symmetric functions, 
defined by $\ch(\chi^{\lambda}) = s_{\lambda}$ and extended by linearity. Our
main result is:

\begin{theorem}\label{conj1}
For any $1\le k\le n$,
the set $Bl_{n,k}$
is Schur-positive.
In fact, for $1 \le k \le n-1$
\[
\Q(Bl_{n,k}) = \ch(\chi^{(n-1,n-k)}
\downarrow_{\S_n}^{\S_{2n-k-1}})
\]
where $\chi\downarrow^G_H$
stands for the restriction of a character $\chi$ from the group $G$ to its subgroup $H$; and, for $k = n$
\[
\Q(Bl_{n,n}) = \ch(\chi^{(n)}) = s_{(n)}.
\]
\end{theorem}

The proof involves a left-to-right-maxima-preserving bijection and a resulting equi-distribution result.
Specifically, let
\[
\Ltr (\pi):=\{i:\ \pi(i) = \max\{\pi(1), \ldots, \pi(i)\}\}
\]
be the set of {\em left-to-right maxima} in a permutation $\pi$,
and let
\[
\ldes(\pi):=\max\{i:\ i\in \Des(\pi)\}
\]
be the {\em last descent} of $\pi$,
with $\ldes(\pi) := 0$ if $\Des(\pi) = \varnothing$
(i.e., if $\pi$ is the identity permutation).

For every $I \subseteq [n],$ let ${\bf x}^I:=
\prod\limits_{i \in I} x_i$.
Our equi-distribution result is:

\begin{theorem}\label{main1_introduction}
For every positive integer $n$
\[
\sum_{\pi\in \S_n(321)}  {\bf x}^{\Ltr(\pi)} q^{\bl(\pi)} =
\sum_{\pi\in \S_n(321)}  {\bf x}^{\Ltr(\pi)}
q^{n-\ldes(\pi^{-1})}.
\]
\end{theorem}

\begin{remark}
An equivalent formulation, replacing $\pi$ by $\pi^{-1}$ and using $\bl(\pi^{-1}) = \bl(\pi)$, is:
\[
\sum_{\pi\in \S_n(321)}  {\bf x}^{\Ltr(\pi^{-1})} q^{\bl(\pi)} =
\sum_{\pi\in \S_n(321)}  {\bf x}^{\Ltr(\pi^{-1})}
q^{n-\ldes(\pi)}.
\]
This is reminiscent of the classical Foata-Sch\"{u}tzenberger Theorem~\cite[Theorem 1]{FS}
\[
\sum_{\pi\in \S_n}  {\bf x}^{\Des(\pi^{-1})} q^{\inv(\pi)} =
\sum_{\pi\in \S_n}  {\bf x}^{\Des(\pi^{-1})}
q^{\maj(\pi)};
\]
see Observation~\ref{obs2} below.
\end{remark}

The proof of Theorem~\ref{main1_introduction}  applies a recursive bijection, that may be considered as a $321$-avoiding analogue of Foata's fundamental bijection~\cite{Foata}. An alternative bijective proof of Theorem~\ref{main1_introduction}, applying Krattenthaler's bijection between $321$-avoiding permutations and Dyck paths, 
is presented in a recent follow-up paper by Martin Rubey~\cite{Rubey}. 

\bigskip

The rest of the paper is organized as follows.
Necessary preliminaries and background are given in Section~\ref{sec:prelim}.
The block decomposition of a permutation is described in Section~\ref{sec:block},
and the enumeration of $321$-avoiding permutations with prescribed block number is discussed in Section~\ref{sec:count}.
The equi-distribution phenomenon (Theorem~\ref{main1_introduction}) is proved in Section~\ref{sec:equidistribution}.
In Section~\ref{sec:Schur} this phenomenon is applied to prove Schur-positivity.
We close the paper with final remarks and open problems.

\section{Preliminaries}\label{sec:prelim}

\subsection{Statistics on permutations and on SYT}\label{subsec:stat}

For a positive integer $n$ let $[n]:=\{1,2,\dots,n\}$, and let $\S_n$ denote the $n$-th symmetric group, the group of all permutations of $[n]$.
The {\em descent set} of a permutation $\pi\in \S_n$ is defined by
\[
\Des(\pi):=\{i \,:\, \pi(i)>\pi(i+1)\}.
\]
For a subset $J\subseteq [n-1]$, define the {\em descent class}
\[
D_{n,J}:=\{\pi\in \S_n \,:\, \Des(\pi)=J\}
\]
and the corresponding {\em inverse descent class} 
\[
D_{n,J}^{-1} := \{\pi^{-1} \,:\, \pi\in D_{n,J}\}.
\]
For a permutation $\pi \in \S_n$ let
\[
\Ltr(\pi):=\{i \,:\, \pi(i) = \max\{\pi(1), \ldots, \pi(i)\}\},
\]
the set of {\em left-to-right maxima} in $\pi$.

\begin{obs}
If $\pi \in \S_n$ then the restriction of $\pi$ to $\Ltr(\pi)$ (as a subsequence of $[n]$) is monotone increasing.
If, moreover, $\pi \in \S_n(321)$ then the restriction of $\pi$ to the complementary subsequence $[n] \setminus \Ltr(\pi)$ is also monotone increasing.
\end{obs}

\begin{obs}\label{obs2}
If $\pi \in \S_n(321)$ then the set $\Ltr(\pi)$ uniquely determines the set $\Des(\pi)$.
Explicitly, for any $1 \le i \le n-1$,
\[
i \in \Des(\pi) \iff i \in \Ltr(\pi) \text{ and } i+1 \not\in \Ltr(\pi).
\]
\end{obs}

\begin{example}
Let $\pi=251348697\in \S_9(321)$, written in one-line notation, i.e., $\pi(1)=2$, $\pi(2)=5$, etc.
The restriction of $\pi$ to the set $\Ltr(\pi)=\{1,2,6,8\}$, viewed as the sequence $1268$, is $2589$, a monotone increasing sequence. 
Similarly, the restriction of $\pi$ to the complementary subsequence $[9] \setminus \Ltr(\pi) = 34579$ is $13467$, which is also monotone increasing. 
The descent set $\Des(\pi)=\{2,6,8\}=\{i\in \Ltr(\pi):\ i+1\not\in \Ltr(\pi)\}$.
\end{example}

For a permutation $\pi\in S_n$ let
\[
\ldes(\pi):=\max\{i:\ i\in \Des(\pi)\},
\]
be its {\em last descent}, with $\ldes(\pi) := 0$ if $\Des(\pi) = \varnothing$
(i.e., if $\pi$ is the identity permutation).


For a skew shape $\lambda/\mu$, let $\SYT(\lambda/\mu)$ be the set
of standard Young tableaux of shape $\lambda/\mu$. We use the
English convention, according to which row indices increase from top to bottom
(see, e.g., \cite[Ch.\ 2.5]{Sagan}). The {\em height} of a standard Young tableau $T$ is the number of rows in $T$.
%
The {\em descent set} of $T$ 
is
\[
\Des(T):=\{i \,:\, i+1 \text{ appears in a lower row of } T \text{ than } i\}.
\]
Its {\em last descent} is
\[
\ldes(T):=\max\{i \,:\, i \in \Des(T)\},
\]
with $\ldes(T) := 0$ if $\Des(T) = \varnothing$.

We shall make use of the Robinson-Schensted-Knuth (RSK) correspondence which maps each permutation $\pi \in \S_n$ to a pair $(P_{\pi}, Q_{\pi})$ of standard Young tableaux of the same shape $\lambda$. A detailed description can be found, for example, in  \cite[Ch.\ 3.1]{Sagan} or in \cite[Ch.\ 7.11]{St2}. A fundamental property of the RSK correspondence is:

\begin{fact}\label{fact:RSK_DES}
For each $\pi\in \S_n$,
$\Des(P_{\pi}) = \Des(\pi^{-1})$
and
$\Des(Q_{\pi}) = \Des(\pi)$.
\end{fact}

\subsection{The $k$-fold Catalan number}\label{sec:prel_Catalan}

Recall the $n$-th Catalan number, defined by
\[
C_n := \frac{1}{n+1} \binom{2n}{n} 
= \binom{2n}{n} - \binom{2n}{n+1} \qquad (n \ge 0),
\]
with generating function
\[
c(x) := \sum_{n=0}^{\infty} C_n x^n = \frac{1-\sqrt{1-4x}}{2x}.
\] 


For each $0 \le k \le n$, the {\em $n$-th $k$-fold Catalan number} $C_{n,k}$ is 
the coefficient of $x^n$ in $(xc(x))^k$.
These numbers are also called {\em ballot numbers}, and form the {\em Catalan triangle}~\cite[A009766]{OEIS}.
As proved by Catalan
himself~\cite{C}, they are given explicitly by
\[
C_{n,k} = \frac{k}{2n-k} \binom{2n-k}{n} 
= \binom{2n-k-1}{n-1} - \binom{2n-k-1}{n} \qquad (1 \le k \le n) 
\]
and $C_{n,0} = \delta_{n,0}$ $(n \ge 0)$;
in particular, $C_{n,1} = C_{n-1}$ for $n \ge 1$.

Among the many interpretations of $C_{n,k}$ one can
mention the number of lattice paths from $(k,1)$ to $(n,n)$, consisting
of steps $(1,0)$ and $(0,1)$, which never go strictly above the line $y=x$; see, e.g., \cite[Cor.\ 16]{T} which uses a slightly different indexing.

The following proposition, reformulating results presented in \cite{CGG, T}, 
relates the $k$-fold Catalan numbers to $321$-avoiding permutations and to standard Young tableaux.

\begin{proposition}\label{prop1} {\rm (\cite{CGG, T})}
For positive integers $1 \le k \le n$,
\[
|\{\pi \in \S_n(321) \,:\, \ldes(\pi^{-1})=n-k\}| = C_{n,k}
\]
and
\[
|SYT(n-1,n-k)| = C_{n,k}.
\]
\end{proposition}

\begin{proof}
Connolly et al.~\cite{CGG} showed that the $n$-th $k$-fold Catalan number $C_{n,k}$ is equal to the number of $123$-avoiding permutations in $\S_n$ with first ascent at position $k$, which equals in turn, 
via reversing the permutations in $\S_n$, 
to the number of $321$-avoiding permutations $\pi \in \S_n$ with $\ldes(\pi) = n-k$. Since $\S_n(321)$ is closed under inversion $(\pi \mapsto \pi^{-1})$, the statistics $\ldes(\pi)$ and $\ldes(\pi^{-1})$ have the same distribution over $\S_n(321)$, proving the first assertion.

As for the second assertion, 
$C_{n,k}$ is the number of lattice paths from $(k,1)$ to $(n,n)$, or equivalently from $(0,0)$ to $(n-1,n-k)$, which never go above the line $y=x$~\cite[Cor.\ 16]{T}. 
The latter, in turn, are in bijection with the standard Young tableaux of shape $(n-1,n-k)$, with horizontal steps in the path corresponding to the entries in the first row of the tableau and vertical steps corresponding to the entries in the second row. 

\end{proof}






\subsection{Symmetric and quasi-symmetric functions}\label{sec:prel_quasi}

\ytableausetup{centertableaux}

A {\em partition} of a positive integer $n$ is a weakly decreasing sequence $\lambda=(\lambda_1,\ldots,\lambda_t)$ of positive integers whose sum is $n$. We denote $\lambda \vdash n$.


Let ${\bf x} := (x_1,x_2,\ldots)$ be an infinite sequence of commuting indeterminates. 
Symmetric and quasi-symmetric functions in ${\bf x}$ can be defined over various (commutative) rings of coefficients, including the ring of integers; for simplicity we define it over the field $\bbq$ of rational numbers. 

\begin{defn}                                 
A {\em symmetric function} in the variables $x_1,x_2,\ldots$ is a formal power series $f({\bf x})\in \bbq[[{\bf x}]]$, of bounded degree, 
such that for any three sequences (of the same length $k$) of positive integers, $(a_1,\ldots,a_k)$, $(i_1,\ldots,i_k)$ and $(j_1,\ldots, j_k)$, such that $i_1,\ldots,i_k$ are distinct and $j_1,\ldots, j_k$ are distinct, the coefficients of $x_{i_1}^{a_1}\cdots x_{i_k}^{a_k}$ and of $x_{j_1}^{a_1}\cdots x_{j_k}^{a_k}$ in $f$ are equal:
\[
[x_{i_1}^{a_1}\cdots x_{i_k}^{a_k}]f = 
[x_{j_1}^{a_1}\cdots x_{j_k}^{a_k}]f.
\]
\end{defn}

Schur functions, indexed by partitions of $n$, form a distinguished basis for $\Lambda^n$, the vector space of symmetric functions which are homogeneous of degree $n$;
see, e.g., \cite[Corollary 7.10.6]{St2}.
A symmetric function in $\Lambda^n$ is {\em Schur-positive} if all the coefficients in its expansion in the basis $\{s_{\lambda} \,:\, \lambda \vdash n\}$ of Schur functions are nonnegative.

The following definition of a quasi-symmetric function can be found in~\cite[7.19]{St2}.
                                                                                                        \begin{defn}                                 
A {\em quasi-symmetric function} in the variables $x_1,x_2,\ldots$  is a formal power series $f({\bf x})\in \bbq[[{\bf x}]]$, of bounded degree, 
such that for any three sequences (of the same length $k$) of positive integers, $(a_1,\ldots,a_k)$, $(i_1,\ldots,i_k)$ and $(j_1,\ldots, j_k)$, where the last two are increasing, the coefficients of $x_{i_1}^{a_1}\cdots x_{i_k}^{a_k}$ and of $x_{j_1}^{a_1}\cdots x_{j_k}^{a_k}$ in $f$ are equal:
\[
[x_{i_1}^{a_1}\cdots x_{i_k}^{a_k}]f = 
[x_{j_1}^{a_1}\cdots x_{j_k}^{a_k}]f
\]
whenever $i_1 < \ldots < i_k$ and $j_1 < \ldots < j_k$. 
\end{defn}
                                                                                                                     Clearly, every symmetric function is quasi-symmetric, but not conversely: $\sum_{i<j}{x_i^2 x_j}$, for example, is quasi-symmetric but not symmetric.

For each subset $D \subseteq [n-1]$ define the {\em
fundamental quasi-symmetric function}
\[
\F_{n,D}(\x) := \sum_{i_1\le i_2 \le \ldots \le i_n \atop {i_j <
i_{j+1} \text{ if } j \in D}} x_{i_1} x_{i_2} \cdots x_{i_n}.
\]

Let $\BBB$ be a (multi)set of combinatorial objects, equipped with a {\em descent map} $\Des: \BBB \to 2^{[n-1]}$ which associates to each
element $b\in \BBB$ a subset $\Des(b) \subseteq [n-1]$. Define the
quasi-symmetric function
\[
\Q(\BBB) := \sum\limits_{b\in \BBB} m(b,\BBB) \F_{n,\Des(b)},
\]
where $m(b,\BBB)$ is the multiplicity of the element $b$ in
$\BBB$. With some abuse of terminology, we say that $\BBB$ is Schur-positive when $\Q(\BBB)$ is.

\medskip

The following key theorem is due to Gessel.

\begin{theorem}{\rm \cite[Theorem 7.19.7]{St2}}\label{G1} 
For every shape $\lambda \vdash n$,
\[
\Q({\SYT(\lambda)})=s_{\lambda}.
\]
\end{theorem}


There is a dictionary relating symmetric functions to characters of the symmetric group $\S_n$. The irreducible characters of $\S_n$ are indexed by partitions $\lambda \vdash n$ and denoted $\chi^\lambda$. 
The {\em Frobenius characteristic map} $\ch$ from class functions on $\S_n$ to symmetric functions is defined by $\ch(\chi^{\lambda}) = s_{\lambda}$, and extended by linearity.
Theorem~\ref{G1} may then be restated as follows:
\[
\ch(\chi^\lambda) = \sum_{T \in SYT(\lambda)} \F_{n,\Des(T)}.
\]

A combinatorial rule for the restriction of irreducible $\S_n$-characters was given by Young~\cite[Theorem 9.2]{James}:

\begin{theorem}\label{BranchingRule} (The Branching Rule)
For $\lambda\vdash n$
\[
\chi^\lambda \downarrow_{\S_{n-1}}^{\S_n} = \sum_{\mu \vdash n-1 \atop |\lambda/\mu|=1}\chi^\mu.
\]
\end{theorem}

Viewing tableaux of shape $\mu$ as tableaux of shape $\lambda$ with the entry $n$ ``forgotten'', the Branching Rule may be restated as
\[
\ch(\chi^\lambda \downarrow_{\S_{n-1}}^{\S_n}) = \sum_{T \in SYT(\lambda)} \F_{n-1,\Des(T) \cap [n-2]}.
\]
Iteration immediately gives

\begin{corollary}\label{quasi_restriction}
For every $\lambda\vdash n$ and $m\le n$
\[
\ch(\chi^\lambda \downarrow_{\S_m}^{\S_n}) = \sum_{T \in \SYT(\lambda)} \F_{m, \Des(T) \cap [m-1]}.
\]
\end{corollary}


Recall the notation $P_\pi$ from Subsection~\ref{subsec:stat}.
For every standard Young tableau $T$ of size $n$, the set
\[
{\mathcal C}_T:=\{\pi\in \S_n:\ P_\pi=T\}
\]
is the {\em Knuth class} corresponding to $T$.
Fact~\ref{fact:RSK_DES} together with Theorem~\ref{G1} imply the following well-known result.

\begin{proposition}\label{Knuth_positive}
Knuth classes are Schur-positive.
\end{proposition}

See~\cite[Theorem 5.5]{Gessel-Reutenauer} and~\cite[Prop. 9.5]{Adin-R}.

\medskip

The following lemma will be used in Section~\ref{sec:Schur}.

\begin{lemma}[{\cite[Lemma 8.1]{ER}}]
\label{prop:first-last}
For every Schur-positive set $A\subseteq \S_n$, the (set) statistics
$\Des$ and $n-\Des$ are equi-distributed over $A$.
\end{lemma}


\section{The block number of a permutation}

\subsection{Definitions}\label{sec:block}

Block decomposition and direct sums of permutations appear in the
study of pattern avoiding classes~\cite{Albert-Atkinson05, AAV}.

\medskip


Let $\pi \in \S_m$ and $\sigma \in \S_n$. The {\em direct sum} of $\pi$ and $\sigma$ is the permutation $\pi \oplus \sigma \in \S_{m+n}$ defined by
\[
\pi \oplus \sigma :=
\begin{cases} 
\pi(i), & \text{if } i \leq m;  \\
\sigma(i-m) + m,  & \text{otherwise.}
\end{cases}
\]

For example, if $\pi = 312$ and $\sigma = 2413$ then $\pi \oplus \sigma = 3125746$; see Figure~\ref{fig:sum}. 

\begin{figure}[h]
\begin{center}
\includegraphics[scale=0.5]{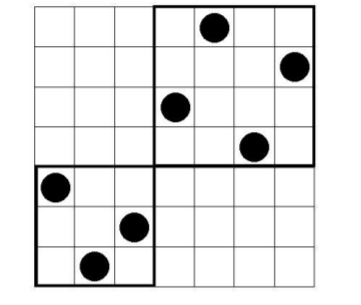}
\caption{The permutation $312\oplus 2413=3125746$ 
}\label{fig:sum}
\end{center}
\end{figure}

A nonempty permutation which is not the direct sum of two nonempty permutations is called {\em $\oplus$-irreducible}. 
Each permutation $\pi$ can be written uniquely as a direct sum of $\oplus$-irreducible ones, called the {\em blocks} of $\pi$; their number, denoted by $\bl(\pi)$, is the {\em block number} of $\pi$.


\begin{exa}
$\bl(45321)=1$, $\bl(31254)=2$, and $\bl(1234)=4$.                               
\end{exa}          


\begin{obs}\label{obs:bl}
The block number of a permutation $\pi\in \S_n$ is 
\[
\bl(\pi)= 1 + |\{1 \le i \le n-1 \,:\, \max(\pi(1), \ldots, \pi(i)) <
\min(\pi(i+1), \ldots, \pi(n))\}|.
\]
\end{obs}

The following equivalent definition was proposed by Michael Joseph and Tom Roby~\cite{personal}.
\begin{obs}
\[
\bl(\pi) = |\{1 \le i \le n \,:\, (\forall j \le i)\, \quad \pi(j) \le i \}|.
\]
\end{obs}




\subsection{Counting 321-avoiding permutations by block number}\label{sec:count}

Recall from Subsection~\ref{sec:prel_Catalan}  
the Catalan generating function
\[
c(x) := \sum\limits_{n=0}^{\infty}{C_n x^n}.
\] 



\begin{proposition}\label{prop:Bl_Catalan}
For any fixed positive integer $k$, the ordinary generating function for the number of $321$-avoiding permutations in $\S_n$ $(n \ge k)$ with exactly $k$ blocks is $(xc(x))^k$.
\end{proposition}

\begin{proof}

Let $d(x)$ be the generating function for $321$-avoiding permutations with exactly one block. 
A permutation is $321$-avoiding if and only if all its blocks are, so that the generating function for $321$-avoiding permutations with exactly $k$ blocks is $d(x)^k$. We therefore only need to show that 
\[
d(x) = x c(x).
\]
Since each permutation has some number $k \ge 0$ of blocks, with $k = 0$ only for the empty permutation, 
we see immediately that
\[
c(x) = 1 + \sum_{k=1}^{\infty} d(x)^k = \frac{1}{1 - d(x)}.
\]
On the other hand, it is well known that $c(x)$ satisfies the equation
\[
x c(x)^2 - c(x) + 1 = 0,
\]
or equivalently
\[
c(x) = \frac{1}{1 - x c(x)}.
\]
Comparison of the two formulas shows that indeed $d(x) = x c(x)$.
\end{proof}

\bigskip
Combining Proposition~\ref{prop:Bl_Catalan} with Proposition~\ref{prop1} we deduce

\begin{corollary}\label{t:cardinalities}
For any integers $1\le k\le n$,
\begin{eqnarray*}
|\{\pi \in \S_n(321) \,:\, \bl(\pi) = k\}| 
&=& \{\pi \in \S_n(321) \,:\, \ldes(\pi^{-1}) = n-k\}| \\
&=& |\SYT(n-1,n-k)|.
\end{eqnarray*}
\end{corollary}


Corollary~\ref{t:cardinalities} will be refined in this paper, see Theorem~\ref{main1_introduction}
and Corollaries~\ref{t:des_equid} and~\ref{t:Bl_SYT} below.

\section{Equi-distribution}\label{sec:equidistribution}

\begin{defn}
For $1 \le k \le n$ denote  
\[
Bl_{n,k}:=\{\pi\in \S_n(321):\ \bl(\pi) = k\}
\]
and
\[
L_{n,k} = \{\pi \in \S_n(321) :\ \ldes(\pi^{-1})=k\}. \]
\end{defn}

In this section we prove Theorem~\ref{main1_introduction} by presenting a left-to-right-maxima-preserving bijection from $Bl_{n,k}$ to $L_{n,n-k}$.


\begin{definition}\label{def:recursion}
Define maps $f_n: \S_n(321) \to \S_n(321)$, recursively, for all $n \ge 1$. For $n = 1$ the definition is obvious, since $\S_1(321)$ consists of a unique permutation. For $\pi \in \S_n(321)$, $n \ge 2$, the recursive definition of $f_n(\pi)$ depends on $k := \bl(\pi)$ and on the locations of the letters $n-1$ and $n$ in $\pi$.
Distinguish the following three cases: 

\begin{description}
\item[{\bf Case A}]
$\pi^{-1}(n)=n$,\ i.e., $n$ is in the last position.\\  
Then: delete $n$, apply $f_{n-1}$, and
insert $n$ at the last position.

\item[{\bf Case B}] 
$\pi^{-1}(n-1)< \pi^{-1}(n) < n$, i.e., $n$ is to the right of $n-1$ but not in the last position.\\ 
Then: delete $n$, apply $f_{n-1}$,  insert $n$ at the same
position as in $\pi$, and multiply on the left by
the transposition $(n-k-1,n-k)$.

\item[{\bf Case C}] 
$\pi^{-1}(n) < \pi^{-1}(n-1)$, i.e.,
$n-1$ is to the right of $n$ (and must be the last letter, since $\pi$ is $321$-avoiding).\\
Then: let $\pi' := (n-1,n)\pi$, define $f_n(\pi')$ according to case A above, and multiply it on the left by the cycle $(n-k,n-k+1, ...,n)$.
\end{description}
\end{definition}

\begin{remark}\label{rem:sequence}
This recursive definition yields a sequence of permutations $(\pi_n,\pi_{n-1},\dots,\pi_1)$, starting with $\pi_n = \pi$.
For each $2 \le i \le n$, $\pi_{i-1} \in \S_{i-1}$ is obtained from $\pi_{i} \in \S_i$ by deleting $i$ from $\pi_{i}$ (in cases A and B) or by deleting $i$ from $(i-1,i)\pi_{i}$ (in case C).
To recover $f_i(\pi_{i})$ from $f_{i-1}(\pi_{i-1})$, the letter $i$ is inserted exactly where it was deleted (for example --- in the last position, in cases A and C), and then the permutation is multiplied, on the left, by a suitable cycle. 
\end{remark}

\begin{example}\label{main example}
Let $\pi=31254786 \in \S_8$, so that $\bl(\pi)=3$ and $\Ltr(\pi)=\{1,4,6,7\}$.
The recursive process is illustrated by the following diagram,
where the arrow $\pi_{i} \to \pi_{i-1}$ is decorated by the case and by the corresponding cycle.
\begin{eqnarray*}
\pi=\pi_8=31254786 &\xrightarrow[(45)]{B}& \pi_7=3125476 \xrightarrow[(4567)]{C} \pi_6=312546 \\
&\xrightarrow[\quad]{A}& \pi_5=31254 \xrightarrow[(345)]{C} \pi_4=3124\xrightarrow[\quad]{A} \pi_3=312 \\
&\xrightarrow[(23)]{C}& \pi_2=21\xrightarrow[(12)]{C} \pi_1=1 .
\end{eqnarray*}


\begin{eqnarray*}
f_1(\pi_1) = 1 &\xrightarrow{(12)}& f_2(\pi_2)=21 \xrightarrow{(23)} f_3(\pi_3)=312 \xrightarrow{\quad} f_4(\pi_4)=3124 \\
&\xrightarrow{(345)}& f_5(\pi_5)=41253 
\xrightarrow{\quad} f_6(\pi_6)=412536 \\
&\xrightarrow{(4567)}& f_7(\pi_7)=5126374 \xrightarrow{(45)} f_8(\pi)=f_8(\pi_8)=41263785.
\end{eqnarray*}


Note that here $\ldes(f_8(\pi)^{-1}) = 5 = 8-\bl(\pi)$ and $\Ltr(f_8(\pi)) = \{1,4,6,7\} = \Ltr(\pi)$.
\end{example}

Our main claim is

\begin{theorem}\label{t:bijection}
For each $1 \le k \le n$, the map $f_n$ defined above is a left-to-right-maxima-preser\-ving bijection from $Bl_{n,k}$ onto $L_{n,n-k}$.  
\end{theorem}

The proof of Theorem~\ref{t:bijection} will proceed through a sequence of lemmas and observations. 
The first observation follows from Remark~\ref{rem:sequence}.

\begin{obs}\label{t:321}
If $\pi \in \S_n(321)$ then $\pi_{n-1} \in \S_{n-1}(321)$.
\end{obs}

This is actually needed for a successful recursion in Definition~\ref{def:recursion}, since $\pi \in \S_n(321)$ is needed to guarantee that $\pi^{-1}(n-1) = n$ in case C. 


\begin{lemma}\label{t:bl}
Let $\pi \in Bl_{n,k}$ ($n \ge 2$).
Then
\[
\pi_{n-1} \in \begin{cases}
Bl_{n-1,k-1}, & \text{if } \pi^{-1}(n) = n \text{ (case A);}\\
Bl_{n-1,k}, & \text{otherwise.}
\end{cases}
\]
\end{lemma}
\begin{proof}
First of all, $\pi_{n-1} \in \S_{n-1}(321)$ by Observation~\ref{t:321}.

If $\pi^{-1}(n) = n$ (case A) then $\{n\}$ is a block of $\pi$, and deleting $n$ reduces the number of blocks by $1$.

If $\pi^{-1}(n-1) < \pi^{-1}(n) < n$ (case B) then $n-1 > \pi(n)$. All the indices from $\pi^{-1}(n-1)$ to the end of the permutation belong to the same block, even after deleting the letter $n$, so that $\bl(\pi_{n-1}) = \bl(\pi)$.

Finally, if $\pi^{-1}(n) < \pi^{-1}(n-1) = n$ (case C) then $\bl[(n-1,n) \pi] = \bl(\pi) + 1$, and deleting $n$ from the last position in $(n-1,n) \pi$ gives $\bl(\pi_{n-1}) = \bl(\pi)$.

\end{proof}

\begin{lemma}\label{t:keep_n}
If $\pi \in \S_n(321)$ then $f_n(\pi)^{-1}(n) = \pi^{-1}(n)$.  
\end{lemma}
\begin{proof}
By induction on $n$. The case $n=1$ is clear. Assume the validity of the claim for $n-1$, and let $\pi \in Bl_{n,k}$. 

If $\pi^{-1}(n-1) < \pi^{-1}(n)$ (cases A and B) then the claim holds by Definition~\ref{def:recursion}, since $n$ is inserted back into the position from which it was deleted and the multiplication by the transposition $(n-k-1,n-k)$ does not affect it.
Note that in case A (and also in case B for $k \ge 2$) we also have $f_n(\pi)^{-1}(n-1) = \pi^{-1}(n-1)$, by the induction hypothesis.

Otherwise $\pi^{-1}(n) < \pi^{-1}(n-1) = n$ (case C). 
We pass from $\pi_n = \pi$ to $\pi_{n-1}$ by interchanging the positions of $n$ and $n-1$ and deleting $n$, so that $\pi^{-1}(n) = \pi_{n-1}^{-1}(n-1)$. By the induction hypothesis we therefore have $f_{n-1}(\pi_{n-1})^{-1}(n-1) = \pi^{-1}(n)$, and after inserting $n$ at the end and multiplying on the left by the cycle $(n-k, \ldots, n-1, n)$ we get $f_n(\pi)^{-1}(n) = \pi^{-1}(n)$. 

\end{proof}

\begin{lemma}\label{t:Bl_to_L}
If $\pi \in Bl_{n,k}$ then $f_n(\pi) \in L_{n,n-k}$; in particular, $f_n(\pi) \in \S_n(321)$. 
\end{lemma}
\begin{proof}
There are actually two claims to prove here: $f_n(\pi) \in \S_n(321)$ and $\ldes(f_n(\pi)^{-1}) = n-k$. We will prove them simultaneously by induction on $n$. The cases $n=1, 2$ are obvious.
Now assume that the claim holds for $n-1$, and let $\pi \in Bl_{n,k}$. 
Consider each of the three cases of Definition~\ref{def:recursion}.

\begin{description}
\item[{\bf Case A}]
$\pi^{-1}(n)=n$.
\\
In this case we have $\pi_{n-1}\in Bl_{n-1,k-1}$ by Lemma~\ref{t:bl}, so by the induction hypothesis 
$f_{n-1}(\pi_{n-1}) \in L_{n-1,n-k}$. Since $f_n(\pi)$ is obtained from $f_{n-1}(\pi_{n-1})$ by inserting $n$ at the last position, we clearly have $f_n(\pi) \in L_{n,n-k}$. 

\item[{\bf Case B}] 
$\pi^{-1}(n-1)< \pi^{-1}(n) < n$. 
\\
In this case, by Lemma~\ref{t:bl}, $\pi_{n-1}\in Bl_{n-1,k}$, so by the induction hypothesis $f_{n-1}(\pi_{n-1}) \in L_{n-1,n-k-1}$. 
Here $f_n(\pi)$ is obtained from $f_{n-1}(\pi_{n-1})$ by inserting $n$ back into its original position and multiplying the permutation on the left by $(n-k-1,n-k)$.
Note that necessarily $k < n-1$, since $k = n$ only for the identity permutation and $k = n-1$ only for adjacent transpositions $\pi = (i,i+1)$ ($1 \le i \le n-1$), which do not fit case B.

Let us first show that $\ldes(f_n(\pi)^{-1}) = n-k$.
Since $n-k-1$ is the last descent of $f_{n-1}(\pi_{n-1})^{-1}$, $n-k$ precedes $n-k-1$ in $f_{n-1}(\pi_{n-1})$ and the letters $n-k, n-k+1, \ldots, n-1$ appear in $f_{n-1}(\pi_{n-1})$ in increasing order.
No letter smaller than $n-k-1$ can appear after $n-k-1$, because $f_{n-1}(\pi_{n-1})$ is $321$-avoiding. Therefore the last letter of $f_{n-1}(\pi_{n-1})$ can be either $n-1$ or $n-k-1$.
In fact, the assumption $\pi^{-1}(n-1) < \pi^{-1}(n) < n$ implies that $n-1$ is not the last letter of $\pi_{n-1}$, and therefore (by Lemma~\ref{t:keep_n}) not the last letter of $f_{n-1}(\pi_{n-1})$; this last letter must therefore be $n-k-1$. Inserting $n$ (to the right of $n-1$ but not at the last position) yields the relative order $n-k,n-k+1,\ldots,n,n-k-1$ for these letters, and multiplying on the left by $(n-k-1,n-k)$ gives $n-k-1,n-k+1,\ldots,n,n-k$. It follows that $\ldes(f_n(\pi)^{-1}) = n-k$, as claimed.

The proof of $321$-avoidance is quite easy: $f_{n-1}(\pi_{n-1})$ is $321$-avoiding by the induction hypothesis; inserting $n$ to the right of $n-1$ does not violate this property, since any violation involving $n$ was already there with $n-1$; and multiplying by $(n-k-1,n-k)$, putting these letters in an increasing order, is certainly good for $321$-avoidance. 

\item[{\bf Case C}] 
$\pi^{-1}(n) < \pi^{-1}(n-1) = n$.
\\
Again, by Lemma~\ref{t:bl},
$\pi_{n-1} \in Bl_{n-1,k}$; and therefore, by the induction hypothesis, $f_{n-1}(\pi_{n-1}) \in L_{n-1,n-k-1}$.
Here $f_n(\pi)$ is obtained from $f_{n-1}(\pi_{n-1})$ by inserting $n$ at the end and multiplying the permutation on the left by $(n-k, n-k+1, \ldots, n)$.
Unlike the previous case, here it {\em is} possible to have $k = n-1$, but only for the transposition $\pi = (n-1,n)$. In that case $\pi_{n-1}$ is the identity permutation, and so is $f_{n-1}(\pi_{n-1})$, leading to $f_n(\pi) = 23 \ldots n 1 \in L_{n,1}$. Having dealt with this case, we shall assume that $k < n-1$.

The proof that $\ldes(f_n(\pi)^{-1}) = n-k$ is quite simple here:
Since $n-k-1$ is the last descent of $f_{n-1}(\pi_{n-1})^{-1}$, 
the letters $n-k, n-k+1, \ldots, n-1$ appear in $f_{n-1}(\pi_{n-1})$ in increasing order.
Inserting $n$ at the end and multiplying the permutation on the left by the cycle $(n-k, n-k+1, \ldots, n)$ gives the order $n-k+1, \ldots, n, n-k$ for the letters $\ge n-k$, and it follows that $\ldes(f_n(\pi)^{-1}) = n-k$.
We note, for later use, that $n-k$ precedes $n-k-1$ in $f_{n-1}(\pi_{n-1})$ (descent of the inverse) and therefore $n-k+1$ precedes $n-k-1$ in $f_n(\pi)$.

Finally, we show that $f_n(\pi)$ is $321$-avoiding: $f_{n-1}(\pi_{n-1})$ is $321$-avoiding by the induction hypothesis, and inserting $n$ at the end preserves this property.
Replacing the order $n-k, n-k+1, \ldots, n$ by the order $n-k+1, \ldots, n, n-k$ can create a $321$-subsequence only if $n-k$ is involved, and this is impossible since $n-k$ is now the last letter and all the larger letters are in increasing order. 

\end{description}

\end{proof}

\begin{lemma}\label{t:bijective}
For each $1 \le k \le n$, the function $f_n$ maps $Bl_{n,k}$ bijectively onto $L_{n,n-k}$.
\end{lemma}
\begin{proof}
By Lemma~\ref{t:Bl_to_L}, $f_n$ maps $Bl_{n,k}$ into $L_{n,n-k}$. Since $\S_n(321)$ is the disjoint union of the sets $Bl_{n,1}, \ldots, Bl_{n,n}$, as well as the disjoint union of $L_{n,n-1}, \ldots, L_{n,0}$, it suffices to prove that $f_n : \S_n(321) \to \S_n(321)$ is injective (and therefore bijective).
This, in turn, can clearly be proved by induction on $n$, using Definition~\ref{def:recursion}, once we show how to recover the case of $\pi$ (A, B or C) from the permutation $\sigma := f_n(\pi)$ alone.

Indeed, Lemma~\ref{t:keep_n} shows that case A occurs if and only if $\sigma^{-1}(n) = n$. 
A careful inspection of the proof of Lemma~\ref{t:Bl_to_L} shows that, if $\sigma^{-1}(n) < n$, case B occurs when $n-k-1$ precedes $n-k+1$ in $\sigma$,
whereas case C occurs when either $k = n-1$ or $n-k+1$ precedes $n-k-1$ in $\sigma$. Note that $k$ can be recovered from $\sigma$ since $\ldes(\sigma^{-1}) = n-k$.

\end{proof}

\begin{lemma}\label{preserve ltr}
Let $\pi \in Bl_{n,k}$. Then $\Ltr(\pi) = \Ltr(f_n(\pi))$
\end{lemma}
\begin{proof}
Again, the proof proceeds by induction on $n$, the cases $n = 1, 2$ being obvious. 
Assume that the claim holds for $n-1$, let $\pi \in Bl_{n,k}$, and consider the three cases of Definition~\ref{def:recursion}.

\begin{description}
\item[{\bf Case A}]
$\pi^{-1}(n)=n$.
\\
By the induction hypothesis, $\Ltr(\pi_{n-1}) = \Ltr(f_{n-1}(\pi_{n-1})) = S$, say.
$f_n(\pi)$ is obtained from $f_{n-1}(\pi_{n-1})$ by inserting $n$ back at the end, yielding
\[
\Ltr(f_n(\pi)) 
= S \cup \{n\}
= \Ltr(\pi).
\]

\item[{\bf Case B}] 
$\pi^{-1}(n-1)< \pi^{-1}(n) < n$.
\\
Denote $i := \pi^{-1}(n)$. The permutation $\pi_{n-1}$ is obtained by deleting $n$ from $\pi$, so that $\Ltr(\pi) = (\Ltr(\pi_{n-1}) \cap [1,i-1]) \cup \{i\}$.
By the induction hypothesis  $\Ltr(\pi_{n-1}) = \Ltr(f(\pi_{n-1}))$, so that inserting $n$ back into its original position recovers $\Ltr(\pi)$.
Finally, by the proof of Lemma~\ref{t:Bl_to_L},
multiplying on the left by $(n-k-1,n-k)$ replaces the order $n-k, n-k+1, \ldots, n, n-k-1$ by the order $n-k-1, n-k+1, \ldots, n, n-k$, again preserving the set of left-to-right-maxima.

\item[{\bf Case C}] 
$\pi^{-1}(n) < \pi^{-1}(n-1) = n$.
\\
Here $\pi_{n-1}$ is obtained by deleting $n$ from $(n-1,n)\pi$, and therefore $\Ltr(\pi_{n-1}) = \Ltr(\pi_n) \subseteq [n-1]$. By the induction hypothesis, this is also $\Ltr(f_{n-1}(\pi_{n-1}))$.
Inserting $n$ at the end of the permutation adds the element $n$ to this set.
Finally, by the proof of Lemma~\ref{t:Bl_to_L},
multiplying on the left by $(n-k, n-k+1,\ldots,n)$ replaces the order $n-k, n-k+1, \ldots, n$ by the order $n-k+1, \ldots, n, n-k$, removing the element $n$ and setting $\Ltr(f_n(\pi)) = \Ltr(\pi)$, as claimed.

\end{description}

\end{proof}



The proof of Theorem~\ref{t:bijection} is now complete.

\qed

\medskip

Theorem \ref{main1_introduction} follows from Theorem \ref{t:bijection}, and implies in turn

\begin{corollary}\label{t:des_equid}
For every positive integer $n$,
\[
\sum_{\pi \in \S_n(321)} {\bf x}^{\Des(\pi)} t^{\pi^{-1}(n)} q^{\bl(\pi)} 
= \sum_{\pi \in \S_n(321)} {\bf x}^{\Des(\pi)} t^{\pi^{-1}(n)}  q^{n-\ldes(\pi^{-1})}.
\]
\end{corollary}

\begin{proof}
By Observation~\ref{obs2}, the set $\Ltr(\pi)$ determines
$\Des(\pi)$. In addition, $\pi^{-1}(n)$ is the maximal element of $\Ltr(\pi)$.
\end{proof}

\section{Schur-Positivity}\label{sec:Schur}

In this Section we prove Theorem~\ref{conj1} and conclude with an application.


\bigskip








\noindent
{\it Proof of Theorem~\ref{conj1}.}
Setting $t=1$ in Corollary~\ref{t:des_equid} gives
\[
\sum_{\pi \in \S_n(321)} {\bf x}^{\Des(\pi)} q^{\bl(\pi)} =
\sum_{\pi \in \S_n(321)} {\bf x}^{\Des(\pi)} q^{n
-\ldes(\pi^{-1})},
\]
and comparing the coefficients of $q^k$ on both sides gives
\[
\Q(Bl_{n,k})=\Q(L_{n,n-k}) \qquad (1 \le k \le n).
\]
It therefore suffices to prove the claim for $L_{n,n-k}$ instead
of $Bl_{n,k}$.

Note that a permutation $\pi$ is $321$-avoiding if and only if
$\text{height}(P_\pi) < 3$, in the terminology of
Subsection~\ref{subsec:stat}. The set
\begin{eqnarray*}
L_{n,n-k} &=& \{\pi \in \S_n(321) \,:\, \ldes(\pi^{-1}) = n-k\} \\
&=& \{\pi \in \S_n \,:\, \text{height}(P_\pi) < 3 \text{ and }
\ldes(P_\pi) = n-k\}
\end{eqnarray*}
is therefore a disjoint union of Knuth classes and thus, 
by Proposition~\ref{Knuth_positive}, Schur-positive.
By 
the above description of $L_{n,n-k}$, Fact~\ref{fact:RSK_DES} and
Proposition~\ref{G1},
\begin{eqnarray*}
\Q(L_{n,n-k}) &=& \sum_{\pi \in
L_{n,n-k}}{\mathcal{F}_{n,\Des(\pi)}}
= \sum_{\substack{\lambda \vdash n \\ \ell(\lambda) < 3}} \,\sum_{\substack{P \in \SYT(\lambda) \\ \ldes(P) = n-k}} \sum_{Q \in \SYT(\lambda)}{\mathcal{F}_{n,\Des(Q)}} \\
&=& \sum_{\substack{\lambda \vdash n \\ \ell(\lambda) < 3}} |\{P
\in \SYT(\lambda) \,:\, \ldes(P) = n-k\}| \, s_{\lambda}.
\end{eqnarray*}
Hence, for every $\lambda \vdash n$ and $1\le k\le n$, the coefficient of $s_\lambda$ in $\Q(L_{n,n-k})$ is
\[
\langle \Q(L_{n,n-k}), s_\lambda \rangle = \begin{cases}
|\{P\in \SYT(\lambda) \,:\, \ldes(P) = n-k\}|, & \text{if } \ell(\lambda) < 3;\\
0, & \text{otherwise}.
\end{cases}
\]

For $k = n$ we have $\ldes(P) = 0$ only for the unique tableau of shape $\lambda = (n)$, so that
\[
\langle \Q(L_{n,0}), s_\lambda \rangle 
= \delta_{\lambda, (n)} 
\]
and therefore
\[
\Q(L_{n,0}) = s_{(n)}
\]
as claimed. 

\medskip

We can therefore assume that $1 \le k \le n-1$, so that $n \le (n-1) + (n-k)$.
Assume first that $\lambda \vdash n$ is not contained in $(n-1,n-k)$.
By the Branching Rule (Theorem~\ref{BranchingRule}),
\[
\langle \chi^{(n-1,n-k)} \downarrow_{\S_n}^{\S_{2n-k-1}},
\chi^\lambda \rangle = 0.
\]
We shall show that in this case we also have 
\[
\langle \Q(L_{n,n-k}), s_\lambda \rangle = 0,
\]
Indeed, a shape $\lambda \vdash n$ not contained in $(n-1,n-k)$ but with $\ell(\lambda) < 3$ has the form $\lambda = (n-m,m)$ with either $m = 0$ or $m > n-k$. A standard Young tableau $P \in \SYT(\lambda)$ cannot have $\ldes(P) = n-k$, since if $m = 0$ then $\ldes(P) = 0$ whereas if $m > n-k$ then $P$ has a descent not smaller than the $m$-th entry in its first row, and this is greater than $n-k$.

\medskip

On the other hand, if $\lambda \vdash n$ is contained in $(n-1,n-k)$ then $\lambda = (n-m,m)$ with $1 \le m \le n-k$. In this case, by the Branching Rule,
\[
\langle \chi^{(n-1,n-k)} \downarrow_{\S_n}^{\S_{2n-k-1}},
\chi^{(n-m,m)}\rangle = |\SYT((n-1,n-k)/(n-m,m))|.
\]
Rotating the shape by $180^\circ$ within a $2 \times n$ box, i.e., applying the transformation $(a,b)/(c,d) \mapsto (n-d,n-c)/(n-b,n-a)$,
the RHS is seen to be equal to
\[
|\SYT((n-m,m)/(k,1))|.
\]
This, in turn, is equal to the number of SYT of shape $(n-m,m)$
with the smallest $k+1$ entries filling a $(k,1)$ shape in some specific order,
say with $1, \ldots, k$ in the first row and $k+1$ in the second. 
These are exactly the SYT of shape $(n-m,m)$ with first descent equal to $k$.
By Lemma~\ref{prop:first-last}, this number is equal to the number of
SYT of shape $(n-m,m)$ with last descent equal to $n-k$, namely to
\[
|\{P \in \SYT(\lambda) \,:\, \ldes(P) = n-k\}| = \langle
\Q(L_{n,n-k}), s_\lambda \rangle.
\]
Thus, in all cases,
\[
\Q(L_{n,n-k}) = \ch(\chi^{(n-1,n-k)}
\downarrow_{\S_n}^{\S_{2n-k-1}}),
\]
as claimed.

\qed


\bigskip

\begin{example}
Let $n=7$ and $k=3$. Then
\[
\langle \chi^{(7-1,7-3)} \downarrow_{\S_7}^{\S_{10}}, \chi^{(5,2)}\rangle =3,
\]
since there are
$3$ standard Young tableaux of skew shape 
\[
(6,4)/(5,2) =
\begin{ytableau}
\none & \none & \none & \none & \none & {} \\
\none & \none & {} & {} 
\end{ytableau}
\]
Rotation by $180^\circ$ within a $2 \times 7$ box gives the skew shape
\[
(5,2) / (3,1) =
\begin{ytableau}
\none & \none & \none & {} & \\
\none & {} 
\end{ytableau}
\]
SYT of this shape correspond bijectively to SYT of shape $(5,2)$ with first descent at $k=3$:
\[
\begin{ytableau}
 1 & 2 & 3 & 6 & 7  \\
 4 &  5 
\end{ytableau}
\, ,\, 
\begin{ytableau}
 1 & 2 & 3 & 5 & 7  \\
 4 &  6
\end{ytableau}
\, , \, 
\begin{ytableau}
 1 & 2 & 3 & 5 & 6  \\
 4 &  7
\end{ytableau}
\]
which, in turn, correspond to SYT of shape $(5,2)$ with last descent at 
$n-k=4$:
\[
\begin{ytableau}
 1 & 2 & 3 & 4 & 7  \\
 5 &  6 
\end{ytableau}
\, , \, 
\begin{ytableau}
 1 & 2 & 4 & 6 & 7  \\
 3 &  5
\end{ytableau}
\, , \, 
\begin{ytableau}
 1 & 3 & 4 & 6 & 7  \\
 2 &  5
\end{ytableau}
\, .
\]
\end{example}

\bigskip

\begin{corollary}\label{t:Bl_SYT}
For any $1 \le k \le n$,
\[
\sum_{\pi \in Bl_{n,k}} {\bf x}^{\Des(\pi)} = \sum_{T \in
\SYT(n-1,n-k)}{\bf x}^{\Des(T) \cap [n-1]}.
\]
\end{corollary}

\begin{proof}
Mapping each monomial $x^D$ to the fundamental quasi-symmetric function $\F_{n,D}$, the claimed equality is equivalent to
\[
\Q(Bl_{n,k}) = \sum_{T \in \SYT(n-1,n-k)} \F_{n, \Des(T) \cap [n-1]}.
\]
For $k = n$, this means
\[
\Q(Bl_{n,n}) = \F_{n, \varnothing}
\]
which is part of Theorem~\ref{conj1}.

For $1 \le k \le n-1$, Corollary~\ref{quasi_restriction} implies that the RHS is equal to the Frobenius image of the restriction
\[
\chi^{(n-1,n-k)} \downarrow_{\S_n}^{\S_{2n-k-1}},
\]
and the claimed equality again follows from Theorem~\ref{conj1}.
\end{proof}

\section{Final remarks and open problems}

\subsection{Hilbert series} 
The quotient $P_n/I_n$ of the polynomial ring $P_n=\mathbb{Q}[x_1,\dots,x_n]$ by the ideal generated by quasi-symmetric functions without constant term 
was studied by Aval, Bergeron and Bergeron~\cite{ABB}, who determined its Hilbert series (with respect to the grading by total degree) in terms of statistics on Dyck paths.
An alternative  description follows from Corollary~\ref{t:des_equid}.

\begin{proposition}
The Hilbert series of the quotient $P_n/I_n$, graded by total degree, is equal to
\[
\sum_{\pi \in \S_n(321)} {q^{n-\bl(\pi)}}.
\]
\end{proposition}

\begin{proof}
By~\cite[Cor. 6.2]{AR}, the Hilbert series of $P_n/I_n$ is equal to 
\[
\sum\limits_{\pi \in \S_n(321)}{q^{\ldes(\pi)}}.
\] 
The set $\S_n(321)$ is closed under inversion, hence 
\[
\sum\limits_{\pi \in \S_n(321)}{q^{\ldes(\pi)}}=\sum\limits_{\pi \in \S_n(321)}{q^{\ldes(\pi^{-1})}}.
\]
Corollary~\ref{t:des_equid} now completes the proof.
\end{proof}


It is now desired to find two different bases for the graded ring  $P_n/I_n$, both indexed by $321$-avoiding permutations,
with total degree equal to the block-number and to the last-descent respectively. 
Determining a nicely behaved linear action of the symmetric group (or of the Temperley-Lieb algebra) on these bases may provide a representation theoretic proof to Corollary~\ref{t:des_equid}. 





\medskip

A natural goal is to generalize the above setting to other reflection groups. 
For type $B$, the framework presented in~\cite{AAER} may be useful.

\subsection{Pattern-statistic Schur-positive pairs} 

Sagan and Woo posed the 
problem of finding Schur-positive pattern-avoiding sets~\cite{Sagan_talk}.
A natural goal is to look further for Schur-positive statistics on pattern-avoiding sets.

\begin{definition}
Let 
$\stat:\S_n \longrightarrow \bbn\cup\{0\}$ be a permutation statistic and $\varnothing \ne \Pi\subseteq \S_m$ be a nonempty set of patterns.
The {\em pattern-statistic pair} $(\Pi,\stat)$ is {\em Schur-positive} if
\[
\Q(\{\pi \in \S_n(\Pi) \,:\, \stat(\pi) = k\})
\]
is Schur-positive for all integers $n \ge 1$ and $k \ge 0$. 
\end{definition}



By Proposition~\ref{Knuth_positive}, sets of permutations which are closed under Knuth relations are Schur-positive. It follows that
if \ $\S_n(\Pi)$ is closed under Knuth relations for every $n\ge 0$, and the statistic $\stat$ is invariant under these relations
then $(\Pi,\stat)$ is a Schur-positive pair. For example, letting $\ides(\pi):=|\Des(\pi^{-1})|$ be the inverse descent number and $e_m$ 
be the identitity permutation in $\S_m$,  the pair $(e_m, \ides)$  
is Schur-positive. For similar reasons, the pair $(\{132,312\},\imaj)$ is Schur-positive, where  
\[
\imaj(\pi) := \sum\limits_{i \in \Des(\pi^{-1})}i
\]
is the inverse major index.


An example of a different type was given in this paper: By Theorem~\ref{conj1}, $(\{321\},\bl)$
is a Schur-poitive pair. Note that the the block number is not invariant under Knuth relations.

\begin{problem}
Find other Schur-positive pattern-statistic pairs, which are not invariant under Knuth relations. 
\end{problem}


\bigskip

\noindent{\bf Acknowledgements:} 
The authors thank 
Tom Roby and Michael Joseph for stimulating discussions and helpful comments, and
 the anonymous referees for valuable remarks and suggestions.
The third author thanks Michael Albert and Mike Atkinson for motivating discussions and support.


\end{document}